\documentclass[11pt]{amsart}

\usepackage{mathptmx} 
\usepackage[scaled=0.90]{helvet} 
\usepackage{courier} 
\normalfont
\usepackage[T1]{fontenc}
\usepackage{color}

\usepackage{hyperref}
\usepackage{amsmath}
\usepackage{amsthm}
\usepackage{amsfonts}
\usepackage{amssymb}
\usepackage{graphicx}
\DeclareGraphicsExtensions{.eps}

\newtheorem{theorem}{Theorem}[section]
\newtheorem{utv*}{Proposition}
\newtheorem{hyp*}{Conjecture}
\newtheorem{lemma}[theorem]{Lemma}

\newtheorem{defin}{Definition}
\newtheorem{zamech}{Remark}
\newtheorem*{th*}{Theorem}

\newcommand{\av}[2]{\langle #1\rangle_{_{\scriptstyle #2}}}

\def\sli{\sum\limits}

\def\la{\lambda}

\def\R{\mathbb{R}}

\def\cEp{\varepsilon}

\def\cD{\mathcal{D}}

\newcommand{\cz}{Calder\'{o}n--Zygmund\ }

\newcommand{\dist}{\operatorname{dist}}
\newcommand{\pd}{\partial}
\newcommand{\om}{\omega}

\newcommand{\eps}{\epsilon}

\newcommand{\cP}{\mathcal{P}}

\newcommand{\cE}{\mathcal{E}}

\newtheorem{cor}[theorem]{Corollary}
\numberwithin{equation}{section}

\theoremstyle{definition}

\def\cyr{\fontencoding{OT2}\fontfamily{wncyr}\selectfont}
\DeclareTextFontCommand{\textcyr}{\cyr}
\newcommand{\sha}[0]{\ensuremath{\mathbb{S}
}}

{\end{list}}

\newcounter{vremennyj}

\begin{document}

\title{The proof of non-homogeneous $T1$ theorem via averaging of dyadic shifts}
\author{Alexander Volberg}
\makeatletter
\@namedef{subjclassname@2010}{
  \textup{2010} Mathematics Subject Classification}
\makeatother




%
%


   \begin{abstract}
We give again a proof of  non-homogeneous $T1$ theorem. Our proof consists of three main parts: a construction of a random ``dyadic'' lattice  as in \cite{NTV}, \cite{NTV2}; an estimate of matrix coefficients of a \cz operator with respect to random Haar basis if a {\it smaller} Haar support is {\it good} like in \cite{NTV2}; a clever averaging trick from ~\cite{Hyt}, \cite{Hyt1} which uses the averaging over dyadic lattices to decompose operator into dyadic shifts eliminating the error term that was present in  \cite{NTV}, \cite{NTV2}.
\end{abstract}

\date{}
\maketitle

\section{Probability space of random dyadic lattices. Non-homogeneous dyadic shifts}
\label{intro}
Consider a positive finite measure $\mu$ supported on a compact set $E$ lying in $\frac12Q_0$, the central cube of the unit cube $Q_0$ of $\R^d$. We recall a construction of probability space of dyadic lattices made in \cite{NTV2}.  Suppose $\cD_N= \cD_N(Q_0)$ denotes the dyadic grid of squares of size $2^{-N}$ in $Q_0$.  We continue this grid to the whole $\R^d$. Now for each such $Q$ there is $2^d$ choices of its father. As soon as one father is chosen, all others are fixed as parallel translation of it. Consider all choices of grids of fathers as equally probable. Now having one grid of fathers fixed, consider $2^d$ choices of their fathers. Choose one independently of the previous choice and again let all choices be equally probable. Continuing like that we build $D_{N-1}(\om)$, $D_{N-2}(\om)$, \dots  In the natural probability space of lattices  $(\Omega_0, \cP_0)$ just built, consider a subset $\Omega$ of lattices such that a cube $Q\in \cD_0$ contains $\frac12Q_0$. Consider $\cP=\frac{\cP_0}{\cP_0(\Omega)} 1_{\Omega}$ and the probability space $(\Omega, \cP)$ is what we will be using now.

In what follows all $\cD=\cD(\omega)=\cup_{k\le N} \cD_k$ are from $(\Omega, \cP)$. So let  $Q$ be in such a $\cD$ and let $Q_i$, $i=1,\dots, 2^d$ are its children.  For any $f\in L^1(\mu)$ we denote $\cE_Q f = \av{f}{1,\mu} 1_Q$,
$$
\cE_k:= \sum_{Q\in \cD_k} \cE_Q f\,,
$$
and 
$$
\Delta_k f = (\cE_{k+1} -\cE_{k})f, \;\Delta_Q f := \Delta_k f \cdot 1_Q\,, Q\in \cD_k\,.
$$
Now let $f\in L_0^2(\mu)$ subscript $0$ meaning that $\int f\, d\mu=0$. Then
$$
f= \sum_{Q\in \cD} \Delta_Q f\,,
$$
and
$$
\Delta_Q f = \sum_{i=1}^{2^d-1} (f, h_Q^i)_\mu h_Q^i\,,
$$
where $h_Q^i$ are called ($\mu-$ Haar functions and the have the following properties
\begin{itemize}
\item $(h_Q^i, h_R^j)_\mu = 0, Q\neq R$,
\item $(h_Q^i, h_Q^j)_\mu = 0, i\neq j$,
\item $\|h_Q^i\|_\mu = 1$,
\item $h_Q^i=\sum_{m=1}^{2^d-1} c_{Q,m} 1_{Q_m}$,
\item $|c_{Q,m}|\le 1/\sqrt{\mu(Q_m)}$.
\end{itemize}
Above $Q_m$'s are children of $Q$.

\begin{defin}
\label{1}
Cube $Q\in \cD(\omega)$ is called {\bf good}  ($(r, \gamma)$-good) if for any $R$ in the same $\cD(\omega)$ but such that
$\ell(R) \ge 2^r \ell(Q)$ one has
\begin{equation}
\label{g}
\dist (Q, sk(R)) \ge \ell(R)^{1-\gamma} \ell(Q)^\gamma\,,
\end{equation}
where $sk(R) := \bigcup_{m=1}^{2^d-1}\pd Q_m$, again $Q_m$'s being all children of $Q$.
\end{defin}

\bigskip

Given $Q\in \cD_k$ we denote $g(Q)=k$. Our main ``tool'' is going to be the famous ``dyadic shifts''. But they will be with respect to non-homogenous measure. Their typical building blocks will be Haar projections with respect to non-homogeneous measure $\mu$.  This is the only slight difference of this note with \cite{Hyt}, \cite{Hyt1}, \cite{HPTV}.

\begin{defin}
Precisely, we call by $\sha_{m,n}$ (shift of complexity $(m,n)$, or shift of complexity $\max (m,n)$) the operator given by the  kernel
$$
f\rightarrow \sum_{L\in \cD} \int_L a_L(x,y)f(y)\,d\mu(y)\,,
$$
where
\begin{equation}
\label{aL}
a_L(x,y) =\sum_{\substack{ Q\subset L, R\subset L\\ g(Q)= g(L)+m, \, g(R)= g(L)+n}}c_{L,Q,R} h_Q^i(x)h_R^j(y)\,,
\end{equation}
where $h_Q^i:=h_Q^{\mu,i}$, $h_R^j:= h_R^{\mu,j} $ are Haar functions (as above) orthogonal and  normalized in $L^2(d\mu)$, and $|c_{L,Q,R}|$ are such that
\begin{equation}
\label{c}
\sum_{Q, R}|c_{L, Q, R}|^2 \le 1\,.
\end{equation}
\end{defin}
Often we will skip superscripts $i, j$. One will always skip superscript $\mu$.

\begin{zamech}
In particular, it is easy to see that if $a_L$ has form \eqref{aL} and satisfies
\begin{equation}
\label{growaL}
|a_L(x,y)|\le \frac1{\mu(L)}\,,
\end{equation}
then \eqref{c} is automatically satisfied, and we are dealing with dyadic shift.
\end{zamech}

\begin{zamech}
A little bit different but basically equivalent definition can be done like that: operator sending $\Delta_L^n(L^2(\mu))$ to itself and having the kernel $a_L(x,y)$ satisfying estimate \eqref{growaL} is called a local dyadic shift of order $n$. Here $\Delta_L^n(L^2(\mu))$ denotes the space of $L^2(\mu)$ functions supported on $L$ and having constant values on children $Q$ of $L$ such that $g(Q)= g(L)+ n+1$.
Now dyadic shift of order $n$ is an operator of the form $\sha_n\,f := \sum_{L\in \cD} \int_L a_L(x,y)f(y)dy$, where $a_L$ corresponds to local shift of order $n$.
\end{zamech}

All these definitions bring us operators satisfying obviously
\begin{equation}
\label{complexity}
\|\sha_{m,n}\|_{L^2(\mu)\rightarrow L^2(\mu)} \le 1\,,
\end{equation}
\begin{equation}
\label{order}
\|\sha_{n}\|_{L^2(\mu)\rightarrow L^2(\mu)} \le n+1\,,
\end{equation}

We also need generalized shifts, but only of complexity $(0,1)$.

\begin{defin}
Let $ \Pi f:= \sum_{L\in \cD} \av{f}{L,\mu} \sqrt{\mu(L)}\sum_{\ell\subset L, |\ell| = 2^{-s} |L|}c_{L,\ell}\cdot h_{\ell}^j$, where $\{c_{L,\ell}\}$ satisfy {\it not just} the condition $\mu(L)\sum_{\ell\subset L, |\ell| = 2^{-s} |L|} |c_{\ell,L}|^2 \le 1$  that would be required for the usual $(0,s)$-shift normalization condition, but a rather stronger Carleson condition
\begin{equation}
\label{carle}
\sum_{L\subset R, \, L\in \cD} \mu(L)\sum_{\ell\subset L, |\ell| = 2^{-s} |L|} |c_{\ell,L}|^2 \le \mu(R)\,.
\end{equation}
Then $\Pi$ is called a generalized shift of complexity $(0,s)$.
\end{defin}

\begin{zamech}
Notice that here $h_L^j$ is a $\mu$-Haar function. But we already skipped $\mu$ and we will be skipping $j$ as well. In fact, a morally same definition would be 
$$
 \Pi f:= \sum_{L\in \cD} \av{f}{L,\mu}\sqrt{\mu(L)}\sum_{\ell\subset L, |\ell| = 2^{-s} |L|}c_{L,\ell} \Delta_\ell\,,
 $$
 where $c_{L,\ell}$ satisfy \eqref{carle}.
 Notice that all these $\Pi$'s are variations of  familiar paraproduct operators. They are of course also bounded in $L^2(\mu)$ by a universal constant (actually by constant $2$).
 \end{zamech}
 
 \section{Main Theorem}
 \label{main}
 
 We are in a position to formulate our main results. Recall when operator $T$ is called an   operator with \cz kernel of order $m$.
 
 \begin{defin}
 \label{3}
Let $X$ be a geometrically doubling metric space.

Let $\lambda(x,r)$ be a positive function, increasing and doubling in $r$, i.e. $\lambda(x, 2r)\leqslant C\lambda(x,r)$, where $C$ does not depend on $x$ and $r$.

We call $K(x,y)\colon X\times X \to \R$  a \cz kernel, associated to a function $\lambda$, if
\begin{align}
&|K(x,y)|\leqslant \min\left( \frac{1}{\lambda(x,|xy|)}, \frac{1}{\lambda(y,|xy|)}\right),\\
& |K(x,y)-K(x',y)|\leqslant \frac{|xx'|^\cEp}{|xy|^\cEp \lambda(x, |xy|)}, \; \;|xy|\geqslant C|xx'|,\\
&|K(x,y)-K(x,y')|\leqslant \frac{|yy'|^\cEp}{|xy|^\cEp \lambda(y, |xy|)}, \; \;|xy|\geqslant C|yy'|.
\end{align}
By $B(x,r)$ we denote the ball in $|.|$ metric, i.e., $B(x,r)=\{y\in X\colon |yx|<r\}$.

Let $\mu$ be a measure on $X$, such that $\mu(B(x,r))\leqslant C\lambda(x,r)$, where $C$ does not depend on $x$ and $r$.
We say that $T$ is an operator with \cz kernel $K$ on our metric space $X$ if
\begin{align}
 Tf(x) = \int K(x,y)f(y)d\mu(y), \; \forall x\not\in \textup{supp}\mu, \; \forall f\in C_0^\infty\,.
\end{align}
We call $T$ a  Calder\'on-Zygmund  operator, if it is an operator with \cz kernel, and on the top of that
$$
 T\;\text{is bounded}\; L^2(\mu)\rightarrow L^2(\mu)\,.
 $$
\end{defin}

Let us notice that $\R^d$ is of course a geometrically doubling metric space with the usual euclidean metric. 
Specify the above definition to
$$
\la(x, r)=r^m,\; |xy|=|x-y|\,,
$$
then we have a \cz kernel of order $m$ in $\R^d$ and \cz operators of order $m$ in $\R^d$. Measure $\mu$ satisfying $\mu(B(x,r))\le Cr^m$ (or more generally $\mu(B(x,r))\leqslant C\lambda(x,r)$) but not satisfying any estimate from below is called {\it non-homogeneous}. The corresponding \cz operator then is also called {\it non-homogeneous}.

\begin{theorem} 
\label{rsh}
Let $\mu(B(x,r)) \le r^m$. Let $T$ be a \cz operator  of order $m$ in $\R^d$. Then there exists a probability space of dyadic lattices $(\Omega, \cP)$ such that
\begin{equation}
\label{deco}
T =c_{1,T}\int_\Omega \Pi(\omega) \,d\cP(\omega) + c_{2,T}\int_\Omega \Pi^*(\omega) \,d\cP(\om) +c_{3,T}\sum_{n=0}^{\infty} 2^{-n\eps_T} \int_\Omega \sha_n(\omega) \, d\cP(\omega)\,.
\end{equation}
Moreover, $\eps_T>0$. Constants $c_{1,T}, c_{2, T}, c_{3,T}$ depend on the \cz parameters of the kernel, on $m$ and $d$, and on the best constant in the so-called $T1$ conditions:
\begin{equation}
\label{t1}
\|T1_Q\|^2_{2, \mu}  \le C_0 \mu(Q)\,,
\end{equation}\begin{equation}
\label{tstar1}
\|T^*1_Q\|^2_{2, \mu}  \le C_0 \mu(Q)\,,
\end{equation}
\end{theorem}

\begin{zamech}
The same thing holds on general geometrically doubling metric space $X$ (not just $\R^d$) and any  non-homogeneous \cz operator 
having \cz kernel in the generalized sense above. Of course measure should satisfy
$$
\mu(B(x,r))\leqslant C\lambda(x,r)\,.
$$
We prefer to prove the $\R^d$-version just for the sake of avoiding some slight technicalities. For example, the construction of the suitable probability space of random dyadic lattice on $X$ is a bit more involved than such construction in $\R^d$. See \cite{NRV}.
\end{zamech}

From this we obtain immediately the non-homogeneous $T1$ theorem:
\begin{cor}
\label{T1cor}
Let operator $T$ be a \cz operator  of order $m$.  Let  $\mu$ be a measure of growth at most $m$: $\mu(B(x,r)) \le r^m$.
Then $\|T\|_{2,\mu} \le C(C_0, d, m, \eps)$, where $\eps$ is from Definition \ref{3}, and $C_0$ is from  $T1$ assumptions \eqref{t1}, \eqref{tstar1}.
\end{cor}

\begin{zamech}
This corollary ($T1$ theorem)  has a long story, if $\mu=m_d$ it was proved by David--Journ\'e. For homogeneous (doubling) measures $\mu$ it was proved by Christ \cite{Chr}.  In the case of  non-homogeneous $\mu$, the proof  is slightly involved,  it has a relatively short exposition in \cite{NTV2}. Here we follow the lines of \cite{NTV2} to prove the decomposition to random dyadic shift Theorem \ref{rsh}. Just \cite{NTV2} is not quite enough for that goal, and we use a beautiful step of Hyt\"onen as well.
Then non-homogeneous $T1$ theorem is just a corollary of the decomposition result, because all shifts of order $n$ involved in \eqref{deco} have norms at most $n+1$ (see the discussion above), but decomposition \eqref{deco} has  an exponentially decreasing factor.
\end{zamech}

\section{Proof of \eqref{deco}}
\label{PrDeco}

Above good cubes were introduced. The cube is {\bf bad} if it is not good.
 
\begin{lemma} 
\label{goodlm}
$\cP\{Q \,\text{is bad}\, \} \le C_1 2^{-c_2r}$.
\end{lemma}

\begin{proof}
By the construction this probability gets estimated by the sum of volumes of $2^{-\gamma s}$, $s\ge r$, neighborhoods of the boundary of a unit cube (the reader can easily understand why the unit cube by using the  scaling invariance).
\end{proof}

\begin{lemma} 
\label{goodlm}
By a small change in probability space we can think that $\cP\{Q \,\text{is good}\, \}$ is independent of $Q$.
\end{lemma}

This would be obvious if we would not pass from  $\Omega_0$ to $\Omega$, no change would be needed.  Otherwise, this is not quite true without the change, but take a cube $Q(\omega)$.  We already know that if $r$ is large and fixed
\begin{equation}
\label{pigood}
\cP\{Q \,\text{is good} \} \ge 1-C_12^{-c_2r}=:p_Q\ge \frac12\,.
\end{equation}
Take a random variable $\xi_{Q}(\omega^{'})$, which is equally distributed on $[0,1]$. We know that
$$
\mathbb{P}(Q \; \mbox{is good}) = p_{Q}\,.
$$
We call $Q$ ``really good'' if
$$
\xi_Q \in [0, \frac{1}{2p_Q}].
$$
Otherwise $Q$ joins bad cubes. Then
$$
\mathbb{P}(Q \; \mbox{is really good}) = \frac12,
$$
and we are done.

\begin{zamech}
We do not want to use ``really good" expression below. But in fact everywhere below when we write ``good" we mean ``really good" in the above sense. We need this only to have the probability of being good the same for all cubes.
\end{zamech}

\bigskip

Let $f, g\in L^2_0(\mu)$, having constant value on each cube from $\cD_N$.  We can write
$$
f= \sli_Q \sli_j(f, h^j_Q)h^j_Q, \;\; \; \; \; g=\sli_R\sli_i (g, h^i_R)h^i_R.
$$

First, we state and proof the theorem, that says that essential part of bilinear form of $T$ can be expressed in terms of pair of cubes, where the smallest one is good.  This is almost what has been done in \cite{NTV2}. The difference is that in \cite{NTV2} an error term (very small) appeared. To eliminate the error term we follow the idea of Hyt\"onen \cite{Hyt}. In fact, the work \cite{Hyt} improved on ``good-bad" decomposition of \cite{NTV}, \cite{NTV2}, \cite{NTV3} by replacing inequalities by an equality and getting rid of the error term.

\begin{theorem}
\label{rand1}
Let $T$ be any linear operator. Then the following equality holds:
$$
\frac12\,\cE \sli_{\stackrel{Q,R,i,j}{\ell(Q)\geqslant \ell(R)}}(Th_Q^j, h_R^i)(f, h_Q^j)(g, h_R^i) = \cE \!\!\!\!\sli_{\stackrel{Q,R,i,j}{\ell(Q)\geqslant \ell(R), \; R \; \mbox{is good}}}\!\!\!\!(Th_Q^j, h_R^i)(f, h_Q^j)(g, h_R^i).
$$
The same is true if we replace $\geqslant$ by $>$.
\end{theorem}

\begin{proof}
We denote
$$
\sigma(T) = \sli_{\ell(Q)\geqslant \ell(R)}(Th^j_Q, h^i_R)(f, h^j_Q)(g, h^i_R).
$$
$$
\sigma'(T)= \sli_{\substack{\ell(Q)\geqslant \ell(R)\\ R \; \mbox{is good}}}(Th^j_Q, h^i_R)(f, h^j_Q)(g, h^i_R).
$$
We would like to get a relationship between $\cE\, \sigma(T)$ and $\cE\,\sigma'(T)$.
We fix $R$ and put 
$$
g_{good}:= \sli_{R \; \mbox{is good}} (g,h^i_R)h^i_R\,.
$$
Then
$$
\sli_{Q} \sli_{R \; \mbox{is good}}(Th^j_{Q}, h^i_R)(f, h^j_Q)(g,h^i_R) = \left(T(f),\sli_{R \; \mbox{is good}} (g,h^i_R)h^i_R\right) =\left(T(f), g_{good}\right)\,.
$$
Taking expectations, we obtain

\begin{equation}
\label{allRgood}
\begin{aligned}
&\cE \sli_{Q,R} (Th^j_{Q}, h^i_R)(f, h^j_Q)(g,h^i_R)\mathbf{1}_{R \,\mbox{is good}} =
 \cE (T(f), g_{good})=\\
 & (T(f),\cE\, g_{good})=\frac12\,(T(f),g)=\frac12\,\cE \sli_{Q,R} (Th^j_{Q}, h^i_R)(f, h^j_Q)(g,h^i_R).
 \end{aligned}
\end{equation}
Next, suppose $\ell(Q)<\ell(R)$. Then the goodness of $R$ does not depend on $Q$, and so
$$
 \frac12\, (Th^j_Q, h^i_R)(f, h^j_Q)(g, h^i_R)=\cE \left((Th^j_Q, h^i_R)(f, h^j_Q)(g, h^i_R)\mathbf{1}_{R \,\mbox{is good}}|Q,R\right) \,.
$$
Let us explain this equality. The right hand side is conditioned: meaning that the left hand side involves the fraction of two numbers: 1) the number of all lattices containing $Q, R$ in it  and such that $R$ (the one that is  larger  by size) is good and 2)  the number of lattices containing $Q, R$ in it. This fraction is exactly $\pi_{good}=\frac12$. The equality has been explained.

 Now we fix a pair of $Q, R$, $\ell(Q)<\ell(R)$, and multiply both sides by the probability that this pair is in the same dyadic lattice from our family. This probability is just the ratio of the number of dyadic lattices in our family containing elements $Q$ and $R$  to the number of all dyadic lattices in our family. After multiplication by this ratio and the summation of all terms with $\ell(Q)<\ell(R)$ we get
finally,
\begin{equation}
\label{QmRgood}
 \frac12\,\cE \sli_{\ell(Q)<\ell(R)} (Th^j_Q, h^i_R)(f, h^j_Q)(g, h^i_R)=\cE \sli_{\ell(Q)<\ell(R)} (Th^j_Q, h^i_R)(f, h^j_Q)(g, h^i_R)\mathbf{1}_{R \,\mbox{is good}} \,.
 \end{equation}
 
 \medskip
 
Now we use first \eqref{allRgood} and then \eqref{QmRgood}:
\begin{equation*}
\begin{aligned}
&\frac12\,\cE \sli_{Q,R} (Th^j_{Q}, h^i_R)(f, h^j_Q)(g,h^i_R) = \cE \sli_{Q,R} (Th^j_{Q}, h^i_R)(f, h^j_Q)(g,h^i_R)\mathbf{1}_{R \,\mbox{is good}} =
 \\
&\cE \sli_{\ell(Q)<\ell(R)}(Th^j_{Q}, h^i_R)(f, h^j_Q)(g,h^i_R)\mathbf{1}_{R \,\mbox{is good}} +
 \cE\sli_{\ell(Q)\geqslant\ell(R)}(Th^j_{Q}, h^i_R)(f, h^j_Q)(g,h^i_R)\mathbf{1}_{R \,\mbox{is good}} =
  \\
 &\frac12\,\cE \sli_{\ell(Q)<\ell(R)}(Th^j_{Q}, h^i_R)(f, h^j_Q)(g,h^i_R) +
\cE\sli_{\ell(Q)\geqslant\ell(R), R\;\mbox{is good}}(Th^j_{Q}, h^i_R)(f, h^j_Q)(g,h^i_R),
\end{aligned}
\end{equation*}
and therefore
\begin{equation}
\label{QbRgood}
\cE\sli_{\ell(Q)\geqslant\ell(R),\; R\;\mbox{is good}}(Th^j_{Q}, h^i_R)(f, h^j_Q)(g,h^i_R) = \frac12\,\cE\sli_{\ell(Q)\geqslant\ell(R)}(Th^j_{Q}, h^i_R)(f, h^j_Q)(g,h^i_R)\,,
\end{equation}
which is the statement of our Theorem.
\end{proof}

\bigskip

Now we skip $i, j$ for the sake of brevity. We have just reduced the estimate of the bilinear form $\sum_{Q, R\in\cD} (Th_Q, h_R)(f, h_Q)(g,h_R)$ to the estimate over {\it all} dyadic lattices in our family, but summing over pairs $Q,R$, where the smaller in size is always {\it good}: $\cE\,\sum_{Q, R\in\cD, \,\text{smaller is good}} (Th_Q, h_R)(f, h_Q)$. Split it to two ``triangular" sums:    $\cE\,\sum_{Q, R\in\cD, \,\ell(R)< \ell(Q),\,R\text{ is good}} (Th_Q, h_R)(f, h_Q)(g, h_R)$ and
$$
\cE\,\sum_{Q, R\in\cD, \,\ell(Q)\le \ell(R),\,Q\text{ is good}} (Th_Q, h_R)(f, h_Q)(g, h_R)\,.
$$
They are basically symmetric, so we will work only with the second sum.

First consider $\sigma_0:=\cE\,\sum_{Q\in\cD, \,\,Q\text{ is good}} (Th_Q, h_Q)(f, h_Q)(g, h_Q)$. We do not care where  $Q$ is good or not and estimate the coefficient $(Th_Q, h_Q)$ in the most simple way. Recall that $h_Q=\sum_{j=1}^{2^d} c_{Q,j} 1_{Q_j}$, where $Q_j$ are children of $Q$. We also remember that $|c_{Q,j}|\le 1/\sqrt{\mu(Q_j)}$.
Estimating
$$
|c_{Q,j}||c_{Q,j'}| |(T1_{Q_j}, 1_{Q_{j'}})| \le  1/\sqrt{\mu(Q_j)}  1/\sqrt{\mu(Q_j)} C_0^2sqrt{\mu(Q_j)} \sqrt{\mu(Q_j)} \le C_0^2
$$
by \eqref{t1}, we can conclude that $\sigma_0/C_0^2$  is actually splits to at most $4^d$ shifts of order $0$.

Similarly we can can work with $\sigma_s=\cE\,\sum_{Q, R\in\cD, \, Q\subset R,\,\ell(Q)= 2^{-s} \ell(R),\,Q\text{ is good}} (Th_Q, h_R)(f, h_Q)(g, h_R)$ for $s=1,\dots, r-1$. We need $r$ to be large, but not too much, it depends on $d$ only, and is chosen in \eqref{pigood}.

\subsection{Decomposition of the inner sum.}
Now we start to work with $\sigma_s, s\ge r$.  Fix a pair $Q,R$, and let $R_1$ be a descendant of $R$ such that $\ell(R_1)=2^r \ell(Q)$.
Consider  the son $R_2$ of $R_1$ that contains $Q$. We know that $Q$ is good, in particular,
$$
\dist (Q, \pd R_2)\ge \dist (Q, sk(R_1)) \ge \ell(R_1)^{1-\gamma} \ell(Q)^\gamma =\ell(R_1) 2^{-r\gamma}= 2\cdot 2^{-r\gamma}\ell(R_2)\,.
$$
Number $\gamma$ will be a small one so 
\begin{equation}
\label{Qfar}
\dist (Q, \pd R_2)\ge 2^{-r} \ell(R_1)=\ell(Q)\,.
\end{equation}
We want to estimate $(Th_Q, h_R)$. 

\begin{lemma}
\label{tief}
Let $Q\subset R$, $S(R)$ be the son of $R$ containing $Q$, and let  $\dist (Q, \pd S(R))\ge \ell(Q)$. Let $T$ be a \cz operator with parameter $\eps$ in \eqref{3}. Then
$$
(Th_Q, h_R) = \av{h_R}{S(R)} ( h_Q, \Delta_Q T^*1)_{\mu} +t_{Q,R}\,,
$$
where
$$
|t_{Q,R}| \le  \int_Q\int_{R\setminus S(R)} \frac{\ell(Q)^\eps}{\dist(t, Q) +
\ell(Q))^{m+\eps}}|h_Q(s)||h_R(t)|\,d\mu(s)|\,d\mu(t)
$$
$$
 \int_Q\int_{\R^d\setminus S(R)} \frac{\ell(Q)^\eps}{\dist(t, Q) +
\ell(Q))^{m+\eps}}|h_Q(s)||\av{h_R}{S(R),\mu}|\,d\mu(s)|\,d\mu(t)\,.
$$
\end{lemma}

\begin{proof}
We write $h_R= h_R\cdot 1_{R\setminus S(R)} + \av{h_R}{S(R), \mu} 1_{S(R)}$.  Then we continue by
$$
h_R=\av{h_R}{S(R), \mu} 1+ h_R\cdot 1_{R\setminus S(R)} - \av{h_R}{S(R), \mu} (1-1_{S(R)})\,.
$$
Then (denoting by $c_Q$ the center of $Q$) we write
$$
(Th_Q, h_R) = \av{h_R}{S(R)} ( h_Q, \Delta_Q T^*1)_{\mu} +\int_{R\setminus S(R)}\int_Q [K(x,y)-K(x, c_Q)] h_Q(y) h_R(x)\,d\mu(y)\,d\mu(x) -
$$
$$
\int_{\R^d\setminus S(R)}\int_Q [K(x,y)-K(x, c_Q)] h_Q(y) \av{h_R}{S(R), \mu}\,d\mu(y)\,d\mu(x)\,.
$$
Now the usual \cz estimate of the kernel finishes the lemma. In this estimate we used \eqref{Qfar}.
\end{proof}

After proving this lemma let us consider two integral terms above separately $t_1:=t_{1,Q,R}:= \int_{R\setminus S(R)}\dots$ and
$t_2:=t_{2,Q,R}:= \int_{\R^d\setminus R}\dots$.

In the second integral we estimate $h_Q$ in $L^1(\mu)$: $\|h_Q\|_{1,\mu} \le \sqrt{\mu(Q)}$, and we estimate
$h_R$ in $L^\infty(\mu)$: $\|h_R\|_\infty \le 1/\sqrt{\mu(S(R))}$.   Integral itself is at most  (recall that $\mu(B(x,r)\le r^m)$)
\begin{equation}
\label{int}
\int_{\R^d\setminus  S(R)}\frac{\ell(Q)^\eps}{(\dist(t, Q) +\ell(Q))^{m+\eps}}\,d\mu(t) \le \frac{\ell(Q)^\eps}{\dist(Q, sk(R))^\eps}\,.
\end{equation}

So if $Q$ is good, meaning that $\dist(Q, sk(R)) \ge \ell(R)^{1-\gamma}\ell(Q)^\gamma$ then \eqref{int} gives us
\begin{equation}
\label{t2qr}
|t_{2, Q, R}| \le \Big(\frac{\mu(Q)}{\mu(S(R))}\Big)^{1/2} \frac{\ell(Q)^{1-\eps\gamma}}{\ell(R)^{1-\eps\gamma}}\,.
\end{equation}

In the first integral we estimate $h_Q$ in $L^1(\mu)$: $\|h_Q\|_{1,\mu} \le \sqrt{\mu(Q)}$, and we {\it cannot} estimate
$h_R$ in $L^\infty(\mu)$: $\|h_R\|_{L^\infty(R\setminus S(R)} $.   The problem is that this supremum is bounded by $1/\sqrt{\mu(s(R))}$ 
for a sibling $s(R)$ of $S(R)$. But because doubling is missing this can 
be an uncontrollably bad estimate. The term $ \Big(\frac{\mu(Q)}{\mu(R)}\Big)^{1/2}$ is a good term , 
at least it is bounded by $1$, on the other hand the term $ \Big(\frac{\mu(Q)}{\mu(R)}\Big)^{1/2}$ 
is not bounded by anything, it is uncontrollable.  
Therefore, we estimate here $\|h_R\|_{1, \mu} \le \sqrt{\mu(R)}$.  
Integral itself we are forced to estimate in $L^\infty$ as $L^1(\mu)$ has been just spent. 
So we get the term $\frac{\ell(Q)^\eps}{\dist(Q, sk(R))^{m+\eps}}$.

So if $Q$ is good, meaning that $\dist(Q, sk(R)) \ge \ell(R)^{1-\gamma}\ell(Q)^\gamma$ then \eqref{int} gives us
$$
|t_{1, Q, R}| \le (\mu(Q)\mu(R))^{1/2} \frac{\ell(Q)^\eps}{\ell(R)^{m+\eps-(m+\eps)\gamma}\ell(Q)^{(m+\eps)\gamma}}\,.
$$
Choose $\gamma:= \frac{\eps}{2(m+\eps)}$. Then we get 
\begin{equation}
\label{t1qr}
|t_{1, Q, R}| \le \Big(\frac{\ell(Q)}{\ell(R)}\Big)^{\eps/2}\frac{\sqrt{\mu(Q)}\sqrt{\mu(R)}}{\ell(R)^m}\le \Big(\frac{\mu(Q)}{\mu(R)}\Big)^{1/2} \Big(\frac{\ell(Q)}{\ell(R)}\Big)^{\eps/2}\,.
\end{equation}
We again used that $\mu(B(x,r)\le r^m)$ in the last inequality. Compare now \eqref{t2qr} and \eqref{t1qr}. We see that for small $\gamma$ (and our $\gamma$ is small) $\eps/2 \le 1-\gamma\eps$, and we can conclude that estimate \eqref{t1qr} holds for both terms $t_{1,Q,R}, t_{2, Q, R}$.

Now notice that sums $\sum_{Q\subset R,\, \ell(Q)=2^{-r-k}\ell(R),\, Q\, \text{is good}} (Th_Q, h_R) (f, h_Q), (g, h_R)$, $k\ge 0$, can be written as three sums:
$$
\sum_R \sum_{Q\subset R,\, \ell(Q)=2^{-r-k}\ell(R),\, Q\, \text{is good}} t_{1,Q,R} (f, h_Q), (g, h_R)\,,
$$
$$
\sum_R\sum_{Q\subset R,\, \ell(Q)=2^{-r-k}\ell(R),\, Q\, \text{is good}} t_{2, Q, R} (f, h_Q), (g, h_R)\,,
$$
and
$$
\sum_R\sum_{Q\subset R,\, \ell(Q)=2^{-r-k}\ell(R),\, Q\, \text{is good}} \av{h_R}{S(R)} ( h_Q, \Delta_Q T^*1)_{\mu} (f, h_Q) (g, h_R)\,.
$$

Obviously, the first sum is the bilinear form of a shift of complexity $(0, r+k)$ having the coefficient $2^{-\frac{\eps(r+k)}{2}}$ in front, just look at \eqref{t1qr}. The second sum is also the bilinear form of a shift of complexity $(0, r+k)$ having the coefficient $2^{-\frac{(1-\eps\gamma)(r+k)}{2}}$ in front. We just look at \eqref{t2qr} and notice that $\sum_{Q\subset S(R)} \Big(\big(\mu(Q)/\mu(S(R)\big)^{1/2} \Big)^2\le 1$.

This is exactly what we need, and the part of $\cE \sum_{\ell(Q)\le \ell(R),\, Q\,\text{is good}}\dots$, which is given by $Q\subset R$
is represented as the sum of shifts of complexity $(0, n)$, $n\ge r$, 
with exponential  coefficients of the form 
$2^{-\delta n}$, $n>0$. However, this is done  up to the third sum.

We cannot  take care of the third sum  individually. Instead we sum the third sums in all $k\ge 0$ and all $h_Q^i, i=1,\dots, 2^d-1$, $h_R^j, j=1,\dots, 2^d-1$  (we recall that the index $i$ was omitted, now we remember it). After summing over $k$ and $i, j$ we get  ($F(L)$ denotes the dyadic father of $L$)
$$
\sum_{i=1}^{2^d-1}\sum_{L\in \cD} (g, h_{F(L)}^i) \av{h_{F(L)}^i}{L, \mu}  \sum_{Q\subset L,\, \ell(Q)\le 2^{-(r-1)}\ell(L),\, Q\,\text{is good}}(\Delta_Q f, \Delta_Q (T^*1))\,.
$$
We introduce the following operator:
$$
\pi (g) :=\sum_{i=1}^{2^d-1}\sum_{L\in \cD} (g, h_{F(L)}^i) \av{h_{F(L)}^i}{L, \mu}  \sum_{Q\subset L,\, \ell(Q)\le 2^{-(r-1)}\ell(L),\, Q\,\text{is good}}\Delta_Q (T^*1)\,.
$$
This is the same as
$$
\pi(g) = \sum_{L\in\cD} \av{\Delta_{F(L)}g}{L,\mu} \sum_{Q\subset L,\, \ell(Q)\le 2^{-(r-1)}\ell(L),\, Q\,\text{is good}}\Delta_Q (T^*1)\,.
$$

We can rewrite this formula by summing  telescopically first over $L\in \cD$ such that $\ell(L)\ge 2^{(r-1)}\ell(Q)$. Then we get (we assume $\int f\, d\mu=0$ for simplicity)
$$
\pi(g)=\sum_{R\in \cD} \av{f}{R}  \sum_{Q\subset R,\, \ell(Q)= 2^{-(r-1)}\ell(R),\, Q\,\text{is good}}\Delta_Q (T^*1)\,.
$$

We check now by inspection the following equality
\begin{equation}
\label{piTstar}
\begin{aligned}
\sum_{i,j=1}^{2^d-1}\sum_{k\ge 0}\sum_{Q\subset R,\, \ell(Q)=2^{-r-k}\ell(R),\, Q\, \text{is good}} &\av{h_R^i}{S(R)} ( h_Q^i, \Delta_Q T^*1)_{\mu} (f, h_Q^i) (g, h_R^j) = 
\\
&(f, \pi(g)) =(\pi^* f, g)\,.
\end{aligned}
\end{equation}

Operator $\pi(g)$ is a $(0, r-1)$ generalized dyadic shift, it is just 
$$
\pi(g)=\sum_{J=1}^{2^d-1}\sum_{R\in \cD} \av{f}{R}  \sum_{Q\subset R,\, \ell(Q)= 2^{-(r-1)}\ell(R),\, Q\,\text{is good}}(h_Q^j,T^*1) h_Q^j
\,.
$$
To see that we need just to check the carleson condition. We fix $L\in \cD$ and we can see by \eqref{Qfar} that the estimate
\begin{equation}
\label{carl1}
 \sum_{Q\subset L,\, \ell(Q)\le 2^{-(r-1)}\ell(L),\, Q\cap \pd L=\emptyset}\|\Delta_Q(T^*1)\|_{\mu}^2 \le C\mu(L)\,.
\end{equation}
 is enough.
 
 To prove \eqref{carl1} we need
 
 \begin{lemma}
 \label{Sprime}
 Let $T$ be \cz operator satisfying \eqref{tstar1}. Let $S$ be any dyadic square, and let $S' =1.1S$. Then
 $$
 \sum_{Q\subset S}\|\Delta_Q T^*1\|_{\mu}^2 \le C_1\mu(S')\,,
 $$
 where $C_1=C(C_0, d, \eps)$ with $C_0$ from \eqref{tstar1}.
 \end{lemma}
 
 \begin{proof}
 We write $\Delta_Q T^*1 = \Delta_Q T^*1_{S'} + \Delta_Q T^*(1-1_{S'})$. The first term is easy:
 $$
 \sum_{Q\subset S}\|\Delta_Q T^*1_{S'}\|_{\mu}^2 \le \|T^*1_{S'}\|_{\mu}^2 \le C_0\mu(S')
 $$
just by \eqref{tstar1}. 

To estimate $ \sum_{Q\subset S}\|\Delta_Q T^*(1-1_{S'})\|_{\mu}^2$ we  fix $f\in L^2(\mu), \|f\|_{\mu}=1$, we write
$$
(f, \Delta_Q T^*(1-1_{S'})= (T\Delta_Q f, 1-1_{S'}) =\int_{\R^d\setminus S'}\int_Q \Big[ K(x, y)-K(x, c_Q)] \Delta_Q f(y)\,d\mu(y)\, d\mu(x)\,.
$$
Such integral we already saw in \eqref{int} because we can estimate it by $\|\Delta_Q f\|_{1,\mu} \le \|\Delta_Q f\|_{1,\mu}\sqrt{\mu(Q)}\le \sqrt{\mu(Q)}$ multiplied by $\int_{\R^d\setminus S'}\frac{\ell(Q)^\eps}{(\dist (x, Q) +\ell(Q))^{m+\eps}}\, d\mu(x)\le C\ell(Q)^\eps/\ell(S)^\eps$.
As $f$ was arbitrary $f\in L^2(\mu), \|f\|_{\mu}=1$, we get that 
$$
\|\Delta_Q T^*(1-1_{S'})\|_{\mu}^2 \le C\mu(Q)\Big(\frac{\ell(Q)}{\ell(S)}\Big)^{2\eps}\,.
$$
Taking the sum over $Q$ we obtain
 $$
 \sum_{Q\subset S}\|\Delta_Q T^*(1-1_{S'})\|_{\mu}^2 \le \|T^*1_{S'}\|_{\mu}^2 \le C_0\mu(S')\,.
 $$
Lemma is proved.
\end{proof}
We proved \eqref{carl1}, but it also proves that $\pi$ is a bounded generalized shift.
 
\subsection{The decomposition of the outer sum}
We are left to decompose 
$$
\cE \sum_{Q\cap R=\emptyset, \,\ell(Q)\le \ell(R),\, Q\,\text{is good}} (Th_Q, h_R)(f, h_Q, g, h_R)
$$
 into the bilinear form of $(s,t)$-shifts with exponentially small in $\max(s,t)$ coefficients.

Denote
$$
D(Q,R):=
\ell(Q) +\dist (Q, R) +\ell(R)\,.
$$
Also let $L(Q,R)$ be a dyadic interval from the same lattice such that $\ell(L(Q,R) \in 2D(Q,R), 4D(Q,R))$ that contains $R$.

Exactly as we did this before we can estimate
$$
(Th_Q, h_R)=\int_R\int_Q [K(x,y) - K(x, c_Q)] h_Q(y)h_R(x)\,d\mu(y)\,d\mu(x)
$$
by estimating $\|h_Q\|_{1, \mu}\le \sqrt{\mu(Q)}, \|h_R\|_{1, \mu}\le \sqrt{\mu(R)}$, and $\frac{\ell(Q)\eps}{\dist(Q, R)^{m+\eps}}\le \ell(Q)^{\eps/2}/\ell(R)^{\eps/2+m}$ if $\dist (Q, R)\le \ell(R)$. Otherwise the estimate is $\ell(Q)^\eps/ D(Q, R)^{m+\eps}$. These two estimates are both united into the following one obviously
\begin{equation}
\label{long}
|(Th_Q, h_R)|\le C\Big(\frac{\ell(Q)}{\ell(R)}\Big)^{\eps/2} \frac{\ell(R)^{\eps/2}}{D(Q,R)^{m+\eps/2}}\sqrt{\mu(Q)}\sqrt{\mu(R)}\,.
\end{equation}

Of course in this estimate we used not only that $Q$ is good, but also that $\ell(Q)\le 2^{-r}\ell(R)$. Only having this latter condition we can apply the estimate on $\dist(Q, R)$ from Definition \ref{1} that was  used in getting \eqref{long}.

However, if $\ell(Q)\in [2^{-r-1}\ell(R), \ell(R)$ we use just a trivial estimate of coefficient
$(Th_Q, h_R)|$, namely
\begin{equation}
\label{triv}
|(Th_Q, h_R)|\le C(C_0, d)\,,
\end{equation}
where $C_0$ is from \eqref{t1}. This is not dangerous at all because such pairs $Q,R$ will be able to form below only  shifts of complexity $(s,t)$, where $0\le s \le t\le r$; the number of such shifts is at most $\frac{r(r+1)}{2}$, and let us recall, that $r$ is not a large number, it depends only on $d$ (see \eqref{pigood}).

Now in a given $\cD\in \Omega$ a pair of $Q,R$ may  or may not be inside $L(Q,R)$ ($R\subset L(Q,R)$ by definition). But the ration of nice lattices (these are those when both $Q,R$ are inside $L(Q,R)$)  with respect to all lattices in which both $Q, R$ are present is bounded away from zero, this ration (probability) satisfies
\begin{equation}
\label{pqr}
p(Q,R) \ge P_d>0\,.
\end{equation}

We want to modify the following expectation
$$
\Sigma:=\cE \sum_{Q\cap R=\emptyset, \,\ell(Q)\le \ell(R),\, Q\,\text{is good}} (Th_Q, h_R)(f, h_Q, g, h_R)\,.
$$
This expectation is really a certain sum itself, namely the sum over all lattices in $\Omega$ divided by $\sharp(\Omega)=: M$. Each time $Q,R$ are not in a nice lattice we put zero in front of corresponding term. This changes very much the sum. However we can make up for that, and we can {\it leave the sum unchanged} if for nice lattices we put the coefficient $1/p(Q,R)$ in front of corresponding terms (and keep $0$ otherwise).

Then
$$
\frac{\text{number of lattices containing}\, Q, R}{M} = \frac1{p(Q,R)}\frac{\text{number of nice lattices containing}\, Q, R}{M}\,.
$$
Notice that the original sum $\Sigma$  terms $Q,R$ multiplied by the LHS.  The modified sum will contain the same terms multiplied by the $RHS$.  So it is not modified at all, it is the same sum exactly! We can write it again as
$$
\cE \sum_{Q\cap R=\emptyset, \,\ell(Q)\le \ell(R),\, Q\,\text{is good}} m(Q,R, \omega)(Th_Q, h_R)(f, h_Q, g, h_R)\,,
$$
where the random coefficients $m(Q,R, omega)$ are either $0$ (if the lattice $\cD=\omega$ is not nice), or $1/p(Q, R)$ if the lattice is nice.

Now let us fix two  positive integers $s\ge t$, fix a lattice, and consider this latter sum only for this lattice, and write it as 
$$
\sum_s\sum_t\sum_{L}\sum_{Q\subset L, R\subset L, \ell(Q)=2^{-t}\ell(L), \ell(R)=2^{-s}\ell(L)} m(P,Q)(Th_Q, h_R) (f, h_Q)(g, h_R)=:\sum_s\sum_t \sigma_{s,t}\,.
$$

Each $\sigma_{s,t}$ is a dyadic shift of complexity $(s,t)$. In fact, use \eqref{pqr} and \eqref{long} and easily see that the sum of squares of coefficients inside each $L$ is bounded (we use again $\mu(B(x, r)\le r^m$). Moreover, the terms $\Big(\frac{\ell(Q)}{\ell(R)}\Big)^{\eps/2}$, $\Big(\frac{\ell(R)}{D(Q,R)}\Big)^{\eps/2}$from \eqref{long} gives us the desired exponentially small coefficient whose size is at most $2^{-\eps (s-t)/2 }\cdot 2^{-\eps s/2} = 2^{-\eps t/2} = 2^{-\eps\max(s,t)/2}$.

Theorem \ref{rsh} is completely proved.


\begin{thebibliography}{XXX}
\label{rf}
\bibitem{Chr} {\sc M. Christ}, A $T(b)$ theorem with remarks on analytic capacity and the Cauchy integral, Colloq. Math. {\bf 60/61} (1990), no. 2, 601--628;
\bibitem{Hyt} {\sc T. Hyt\"{o}nen}, The sharp weighted bound for general Calderon-Zygmund operators, Ann. of Math. (2) 175 (2012), no. 3, 1473--1506;

\bibitem{Hyt1} {\sc T. Hyt\"{o}nen}, Representation of singular integrals by dyadic operators, and the $A_2$ theorem, arXiv: 1108,5119v1, pp. 1--32.






\bibitem{HPTV} {\sc T. Hyt\"{o}nen, C. P\'erez, S. Treil, A. Volberg}, Sharp weighted estimates for dyadic shifts and the $A_2$ conjecture, arXiv:1010.0755v2, to appear in J. f\"ur die reine und angew. Math.





\bibitem{NRV} {\sc F. Nazarov, A. Reznikov, A. Volberg}, The proof of $A_2$ conjecture in a geometrically doubling metric space, arXiv: 1106.1342, pp. 1--23. To appear.

\bibitem{NTV} {\sc F. Nazarov, S. Treil, A. Volberg}, The $Tb$-theorem on non-homogeneous spaces, Acta Math., 190 (2003), 151--239;

\bibitem{NTV2} {\sc F. Nazarov, S. Treil, A. Volberg}, Cauchy integral and Calderon-Zygmund operators on Nonhomogeneous spaces, Int. Math. Research Notices, 1997, No. 15, 703--726;

\bibitem{NTV3} {\sc F. Nazarov, S. Treil, A. Volberg}, Two weight inequalities for individual Haar multipliers and other well localized operators, Math. Res. Lett. {\bf 15} (2008), no. 4, 583--597;


\bibitem{NTV7} {\sc F.~Nazarov, S.~Treil, and A.~Volberg,}
Two weight estimate for the Hilbert transform and
corona decomposition for non-doubling measures, Preprint 2004, arXiv: 1003.1596v1.






\end{thebibliography}
\end{document}